\documentclass[review]{elsarticle}

\usepackage{lineno,hyperref}
\modulolinenumbers[5]

\journal{Journal of \LaTeX\ Templates}
\usepackage{amsfonts,bm}
\usepackage{graphics,amssymb,amsthm,amsmath,epsf}
\usepackage{mathrsfs}
\usepackage{graphicx}
\usepackage[bf,hypcap]{caption}

\biboptions{authoryear}

\numberwithin{equation}{section}

\newtheorem{theorem}{Theorem}[section]

\newtheorem{proposition}[theorem]{Proposition}

\newtheorem{definition}[theorem]{Definition}
\newtheorem{example}{Example}[section]

\newtheorem{remark}{Remark}[section]
\allowdisplaybreaks

\begin{document}

\begin{frontmatter}

\title{Cumulative Residual Extropy of Minimum Ranked Set Sampling with Unequal Samples}


\author[first]{Mohammad Reza Kazemi}

\author[second]{Saeid Tahmasebi}

\author[third]{Camilla Cal\`{i}}

\author[third]{Maria Longobardi\corref{mycorrespondingauthor}}
\cortext[mycorrespondingauthor]{Corresponding author}
\ead{malongob@unina.it}

\address[first]{Department of Statistics, Faculty of Science, Fasa University, Fasa, Iran}
\address[second]{Department of Statistics, Persian Gulf University , Bushehr, Iran}
\address[third]{Department of Biology, University of Napoli Federico II, Napoli, Italy}

\begin{abstract}
Recently, an alternative measure of uncertainty called cumulative residual extropy (CREX) was proposed by \citet{ja-et-al-19}. In this paper, we consider uncertainty measures of minimum ranked set sampling procedure with unequal samples (MinRSSU) in terms of CREX and its dynamic version and we compare the uncertainty and information content of CREX based on MinRSSU and simple random sampling (SRS) designs. Also, using simulation, we study on new estimators of CREX for MinRSSU and SRS designs in terms of bias and mean square error. Finally, we provide a new discrimination measure of disparity between the distribution of MinRSSU and parental data SRS.
\end{abstract}

\begin{keyword}
\texttt{Cumulative residual extropy, Discrimination measure, Minimum ranked set sampling, Stochastic ordering.}
\MSC[2010] 62B10 \sep 60E15 \sep 62D05 \sep 94A17
\end{keyword}

\end{frontmatter}

\linenumbers

\section{Introduction}
Ranked set sampling (RSS) design is a cost-effective sampling for situations where taking actual measurements on units is expensive but ranking units is easy. For the first time, based on the RSS sampling design, \citet{mcIntyre-52} provided a more efficient estimator of the population mean comparing to the simple random sampling (SRS) counterpart.
To learn more about this concept, the readers can refer to \citet{patil-et-al-99}.
There are many available studies that have developed and generalized the method of sampling used in RSS scheme and they efficiently estimate the population parameter comparing to the SRS scheme.
Recently, \citet{qiu-ef-20} studied information content of minimum ranked set sampling procedure with unequal samples (MinRSSU) as useful modification of RSS procedure in terms of extropy. In the MinRSSU, we draw $m$ simple random samples, where the size of the $i$th samples is $i$, $i=1,...,m$. The one-cycle MinRSSU involves an initial ranking of $m$ samples of size $m$ as follows:
\begin{equation*}
\begin{array}{ccccccc}
1: & \underline{\boldsymbol{X_{(1:1)1}}} & & & & \rightarrow & \tilde{X}%
_{1}=X_{(1:1)1} \\
2: &\underline{\boldsymbol{ X_{(1:2)2}}} & \ X_{(2:2)2} & & & \rightarrow
& \tilde{X}_{2} =X_{(1:2)2} \\
\vdots & \vdots & \vdots & \ddots & \vdots & \vdots & \vdots \\
m: &\underline{\boldsymbol{ X_{(1:m)m}}} & X_{(2:m)m} & \cdots & X_{(m:m)m}
& \rightarrow & \tilde{X}_{m}=X_{(1:m)m} \\
\end{array}
\end{equation*}
where $X_{(i:i)j}$ denotes the $i$th order statistic from the $j$th SRS of size $i$. The resulting sample is called one-cycle MinRSSU of size $m$ and denoted by $\boldsymbol{X}^{(m)}_{MinRSSU}=\{\tilde{X}_{i}, i=1,\dots,m\}$. The parameter $m$ should be kept small because the ranking should not be difficult in this sense and the ranking may be done for example by using an easily measurable covariate, then it is not difficult to identify the minimum of ranked individuals in each subset. Note that $\tilde{X}_{i}$ has the same distribution as $X_{(1)i}$
which is the smallest order statistic in a set of size $i$ with probability density function (pdf) $f_{(1)i}(x)=if(x)[1-F(x)]^{i-1}$ and survival function $\bar{F}_{(1)i}(x)=[1-F(x)]^{i}=\bar{F}^{i}(x)$, where $f(.)$, $F(.)$ and $\bar{F}(.)$ are the underlying pdf, cumulative distribution function (cdf) and survival function.
In reliability theory, $\tilde{X}_{i}$ measures the lifetime of a series system.\\
Several authors have worked on measures of information for RSS and its variants. \citet{jaf-ahm-14}
explored the notions of information content of RSS data and compared them with their counterparts in SRS data. \citet{tah-et-al-16}
obtained some results of residual (past) entropy for ranked set samples. \citet{es-et-al-16} studied information measures for record ranked set sampling.
\citet{es-di-ta-18} considered information measures of maximum ranked set sampling procedure with unequal samples in terms of Shannon entropy, R\'enyi entropy and Kullback-Leibler information, instead \citet{tah-et-al-20} in terms of Tsallis entropy.
More recently, \citet{qiu-ef-20} studied information content of MinRSSU in terms of extropy and \citet{raq-qiu-19} considered the problems of uncertainty and information content of RSS data based on extropy measure and the related monotonic properties and stochastic comparisons.\\
Let $X$ denotes a continuous random variable with pdf $f$.
\citet{lad-15} introduced a new measure termed by extropy associated with $X$ as
\begin{equation} \label{HX}
J(X)=-\frac{1}{2}\int_{-\infty}^{+\infty}[f(x)]^{2}dx=-\frac{1}{2}%
\int_{0}^{1}f(F^{-1}(u))du,
\end{equation}
where $F^{-1}(.)$ is the quantile function of $X$. \citet{qiu-17} explored some characterization results, monotone properties, and lower bounds of extropy of order statistics and record values. Also, \citet{qiu-ef-20} and \citet{raq-qiu-19} considered the information measure of extropy $J(X)$ based on MinRSSU and RSS schemes, respectively and compared the results with their counterpart under SRS design.\\
This paper is organized as follows: Section \ref{sec-crj} deals with the results of cumulative residual extropy (CREX) for MinRSSU data by comparing to its counterpart under SRS data. In Section \ref{sec-empirical}, new estimators are proposed for CREX in SRS and MinRSSU designs using empirical approach. Also, by using simulation study, the behavior of estimators of CREX in MinRSSU and SRS are compared in terms of bias and mean square error. Furthermore, we show that how MinRSSU scheme can efficiently reduce the uncertainty measure comparing to SRS design. Section \ref{sec-discimin} provides a new discrimination measure of disparity between the
distribution of MinRSSU and parental data SRS. Section \ref{sec-conclude} concludes the paper.
\section{Cumulative residual extropy of MinRSSU}\label{sec-crj}
Let $X$ denotes the lifetime of a system with survival function $\bar{F}$. Recently, a new measure of information is proposed by \citet{ja-et-al-19} with substituting the function $\bar{F}$ in extropy formula \eqref{HX}. This new measure is called CREX and defined as
\begin{eqnarray} \label{JX}
{\mathcal{\xi J}}(X)=-\frac{1}{2}\int_{0}^{+\infty}\bar{F}^{2}(x)dx.
\end{eqnarray}
Note that $-\infty<{\mathcal{\xi J}}(X)\leq 0$. If the CREX of $X$ is less than that of another random variable, say $Y$, i.e. ${\mathcal{\xi J}}(X)\leq {\mathcal{\xi J}}(Y)$, then $X$ has less uncertainty than $Y$. Now let ${\mathcal{\xi J}}(X)<+\infty$.
Then, for the MinRSSU and SRS designs, we have
\begin{eqnarray} \label{eq-cj-mrssu}
{\mathcal{\xi J}}(\boldsymbol{X}^{(m)}_{MinRSSU})&=&-\frac{1}{2}\prod
\limits_{i=1}^{m}[-2{\mathcal{\xi J}}(X_{(1:i)})]= -\frac{1}{2}\prod \limits_{i=1}^{m}\int_{0}^{+\infty}\bar F^{2i}(x)dx\\
&=& -\frac{1}{2}\prod \limits_{i=1}^{m}\int_{0}^{1}\frac{(1-u)^{2i}}{f(F^{-1}(u))%
}du \notag \\
&=& -\frac{1}{2}\prod \limits_{i=1}^{m}\mathbb{E}\left[\frac{(1-U)^{2i}}{%
f(F^{-1}(U))}\right],
\end{eqnarray}
and
\begin{equation} \label{eq-cj-srs}
\mathcal{\xi J}(\boldsymbol{X}^{(m)}_{SRS})=-\frac{1}{2}\left[\int_{0}^{+\infty}\bar F^{2}(x)dx\right]^{m}=-\frac{1}{2}[-2 {\mathcal{\xi J}}(X)]^{m}.
\end{equation}
To compare the above measures, let us consider the following examples.
\begin{example}\label{ex-2-1}
If $U \sim Uniform(0,1)$, then
\begin{equation}
{\mathcal{\xi J}}(\boldsymbol{U}^{(m)}_{MRSSU})=-\frac{1}{2}\prod
\limits_{i=1}^{m}\frac{1}{2i+1}=-\frac{1}{2}\left(\frac{\sqrt{\pi}}{2^{m}\Gamma(m+\frac{1}{2})}\right)<{\mathcal{\xi J}}(\boldsymbol{U}%
^{(m)}_{SRS})=-0.5\left(\frac{1}{3}\right)^{m}.
\end{equation}
\end{example}

\begin{example}\label{ex-2-2}
If $Z$ is exponentially distributed with mean $\frac{1}{\lambda}$. Then, we have
\begin{equation}
{\mathcal{\xi J}}(\boldsymbol{Z}^{(m)}_{MRSSU})=-\frac{1}{2}\prod
\limits_{i=1}^{m}\frac{1}{2i\lambda}<{\mathcal{\xi J}}(\boldsymbol{Z}%
^{(m)}_{SRS})=-\frac{1}{2}\left(\frac{1}{2\lambda}\right)^{m}.
\end{equation}
\end{example}

\begin{example}\label{ex-2-3}
Let $X$ is finite range distribution with $\bar{F}(x)=(1-ax)^{b}, \;
0<x<\frac{1}{a}, \; a>0, \;b>0$. Then, we have
\begin{equation}
{\mathcal{\xi J}}(\boldsymbol{X}^{(m)}_{MRSSU})=-\frac{1}{2}\prod \limits_{i=1}^{m}%
\frac{1}{a(1+2ib)}<{\mathcal{\xi J}}(\boldsymbol{X}%
^{(m)}_{SRS})=-\frac{1}{2}\left(\frac{1}{a(1+2b)}\right)^{m}.
\end{equation}
\end{example}
\begin{theorem}\label{th-2-1}
Let $\boldsymbol{X}^{(m)}_{MinRSSU}$ be the MinRSSU from
population $X$ with pdf $f$ and cdf $F$. Then, ${\mathcal{\xi J}}(\boldsymbol{X}%
^{(m)}_{MinRSSU})\leq{\mathcal{\xi J}}(\boldsymbol{X}^{(m)}_{SRS})$ for $m>1$.
\end{theorem}
\begin{proof}
Since $\bar{F}^{2}(x)\geq \bar{F}^{2i}(x)$ for $i\geq1$, we have
\begin{equation*}
\left(\int_{0}^{+\infty}\bar{F}^{2}(x)dx\right)^{m}\leq
\prod_{i=1}^{m}\int_{0}^{+\infty}\bar{F}^{2i}(x)dx.
\end{equation*}
The proof follows by recalling \eqref{eq-cj-mrssu} and \eqref{eq-cj-srs}.
\end{proof}

\begin{remark}
If $f(F^{-1}(u))\geq1$, $0<u<1$, then ${\mathcal{\xi J}}(\boldsymbol{X}%
^{(m)}_{MinRSSU})$ is increasing in $m\geq 1$.
\end{remark}
\begin{proof}
From \eqref{eq-cj-mrssu}, we get
\begin{eqnarray*}
\frac{{\mathcal{\xi J}}(\boldsymbol{X}^{(m+1)}_{MinRSSU})}{{\mathcal{\xi J}}(%
\boldsymbol{X}^{(m)}_{MinRSSU})}=\int_{0}^{1}\frac{(1-u)^{2m+2}}{f(F^{-1}(u))}%
du\leq\frac{1}{2m+3}\leq 1.
\end{eqnarray*}
The result follows readily, since the extropy is negative .
\end{proof}
In the following, we provide some results on the cumulative residual extropy of $\boldsymbol{X}^{(m)}_{MinRSSU}$ in terms of stochastic ordering properties. Now, we state important properties of ${\mathcal{\xi J}}(\boldsymbol{X}^{(m)}_{MinRSSU})$ using the stochastic ordering. For that we present the following definitions:
\begin{definition} (Shaked and Shanthikumar, 2007)
Let $X$ and $Y$ be two non-negative random variables with pdfs $f$ and $g$, cdfs $F$ and $G$, and hazard functions $\lambda_{X}(x)=\frac{f(x)}{\bar F(x)}$ and $\lambda_Y(y)=\frac{g(y)}{\bar G(y)}$, respectively. Then\\
1. $X$ is said to be smaller than $Y$ in the usual stochastic order (denoted by $X\leq _{st}Y$) if $P(X\geq x)\leq
P(Y\geq x)$ for all $x\in \mathbb{R}$.\\
2. $X$ is smaller than $Y$ in the hazard rate order (denoted by $X \leq _{hr}Y$) if $\lambda_{X}(x)\geq\lambda_{Y}(x)$ for all $x$.\\
3. $X$ is smaller than $Y$ in the dispersive order (denoted by $X \leq _{disp}Y$) if $f(F^{-1}(u))\geq g(G^{-1}(u))$ for all $u\in(0,1)$,
where $F^{-1}$ and $G^{-1}$ are right continuous inverses of $F$ and $G$, respectively.\\
4. $X$ is said to have decreasing failure rate (DFR) if $\lambda_{X}(x)$ is decreasing in $x$.\\
5. $X$ is smaller than $Y$ in the convex transform order (denoted by $X \leq _{c}Y$) if $G^{-1}F(x)$ is a convex function on the support of $X$.\\
6. $X$ is smaller than $Y$ in the star order (denoted by $X \leq _{*}Y$) if $\frac{G^{-1}F(x)}{x}$ is increasing in $%
x\geq0$.\\
7. $X$ is smaller than $Y$ in the superadditive order (denoted by $X \leq _{su}Y$) if $G^{-1}(F(t+u))\geq G^{-1}(F(t))+G^{-1}(F(u))$ for $t\geq
0, u\geq0$.
8. $X$ is said to have an increasing reversed hazard rate (IRHR) if $\tilde{\lambda}_{X}(x)=\frac{f(x)}{F(x)}$ is increasing in $x$.
\end{definition}
\begin{theorem}
If $X\leq _{st}Y$, then ${\mathcal{\xi J}}(\boldsymbol{X}^{(m)}_{MinRSSU})\geq{\mathcal{\xi J}}(\boldsymbol{Y}^{(m)}_{MinRSSU}),\;m>1.$
\end{theorem}
\begin{proof} By the assumption of the stochastic order, $%
\bar{F}^{2i}(x)\leq \bar{G}^{2i}(x)$ for all $x\geq0$.
Now using \eqref{eq-cj-mrssu}, for $m>1$, we get the desired result.
\end{proof}

\begin{theorem}
Let $X$ and $Y$ be two non-negative random variable.
If $X\leq _{disp}Y$, then ${\mathcal{\xi J}}(\boldsymbol{X}%
^{(m)}_{MinRSSU})\geq {\mathcal{\xi J}}(\boldsymbol{Y}^{(m)}_{MinRSSU})$ for $m>1$.
\end{theorem}

\begin{proof} By the assumption of the dispersive order , $f(F^{-1}(u))\geq
g(G^{-1}(u))$ for all $u\in(0,1)$. Using \eqref{eq-cj-mrssu}, for $m>1$ the
result follows.
\end{proof}

\begin{theorem}
If $X\leq _{hr}Y$, and $X$ or $Y$ is DFR, then ${\mathcal{\xi J}}(\boldsymbol{X}%
^{(m)}_{MinRSSU})\geq {\mathcal{\xi J}}(\boldsymbol{Y}^{(m)}_{MinRSSU})$ for $m>1$.
\end{theorem}
\begin{proof}
If $X\leq _{hr}Y$, and $X$ or $Y$ is DFR, then $X\leq
_{disp}Y$, due to \citet{bag-koc-86}. Thus, from Theorem (2.4) the
desired result follows.
\end{proof}

\begin{theorem}
Let $X$ and $Y$ be two non-negative random variable with pdf's $f$ and $g$,
respectively, such that $f(0)\geq g(0)>0$ . If $X\leq _{su}Y(X\leq _{*}Y \; or \;X\leq _{c}Y)$ , then ${%
\mathcal{\xi J}}(\boldsymbol{X}^{(m)}_{MinRSSU})\geq {\mathcal{\xi J}}(\boldsymbol{Y}%
^{(m)}_{MinRSSU})$ for $m>1$.
\end{theorem}
\begin{proof}
If $X\leq _{su}Y(X\leq _{*}Y \; or \;X\leq _{c}Y)$, then $%
X\leq _{disp}Y$, due to \citet{ah-al-bar-86}. So, from Theorem (2.4) the
desired result follows.
\end{proof}
\begin{proposition}
Let $\boldsymbol{X}^{(m)}_{MinRSSU}$ and $\boldsymbol{X}^{(m)}_{RSS}$ be MinRSSU and RSS data from distribution $X$ with DFR ageing property, respectively. Then for $m>1$ we have
$${%
\mathcal{\xi J}}(\boldsymbol{X}^{(m)}_{MinRSSU})\geq {\mathcal{\xi J}}(\boldsymbol{X}%
^{(m)}_{SRS}).$$
\end{proposition}
\begin{remark}
Let $\varphi$ be a non-negative function such that the derivative $\varphi'(x)\geq1$ for all $x$. Then $X\leq _{disp}\varphi(X)$. Thus,
$${%
\mathcal{\xi J}}(\boldsymbol{X}^{(m)}_{MinRSSU})\geq {\mathcal{\xi J}}(\varphi(\boldsymbol{X})%
^{(m)}_{MinRSSU}).$$
\end{remark}

If $X\leq_{disp}Y$, Theorem 3.B.26 in \citet{sha-shan-07} claims that
$X_{i:m}\leq_{disp}Y_{i:m},\;\; i=1,2,...,m$. Thus, according to Theorems 4.5 and 4.6 in \citet{qiu-17}, we obtain the following Proposition .
\begin{proposition}
Let $\boldsymbol{X}^{(m)}_{MinRSSU}$ be a sample from MinRSSU design.

\noindent (i) If $X$ is DFR ageing property, then ${\mathcal{\xi J}}(X_{(1)m})$ is increasing in $m\geq1$.

\noindent (ii) If $X$ is IRHR property, then ${\mathcal{\xi J}}(X_{(1)m})$ is decreasing in $m\geq1$.
\end{proposition}
\begin{proposition}
Let $\boldsymbol{Y}^{(m)}_{ MinRSSU}=a\boldsymbol{X}^{(m)}_{ MinRSSU}+b$ with $%
a>0$ and $b\geq 0$. Then, ${\mathcal{\xi J}}(\boldsymbol{Y}^{(m)}_{ MinRSSU})=a{%
\mathcal{\xi J}}(\boldsymbol{X}^{(m)}_{ MinRSSU})$.
\end{proposition}
\begin{proposition}
Let $X$ be a symmetric random variable with respect to the finite mean $\mu=%
\mathbb{E}(X)$.
Then
\begin{eqnarray*}
{\mathcal{\xi J}}(\boldsymbol{X}^{(m)}_{ MinRSSU})={\mathcal{CJ}}(\boldsymbol{X}%
^{(m)}_{ MinRSSU}),
\end{eqnarray*}
where ${\mathcal{CJ}}(X)=-\frac{1}{2}\int_{0}^{+\infty}[F_{X}(x)]^{2}dx
$ is the cumulative extropy (see \citet{ja-et-al-19}).
\end{proposition}

Let $X$ be the random lifetime of a system, recall that $%
X_{[t]}=[X-t\mid X\leq t]$
describes the residual lifetime of a system. For all $t\geq 0$ the mean residual lifetime is given by
\begin{eqnarray*}
\mu(t)=E[X-t\mid X\geq t]=\frac{1}{\bar{F}(t)} \int_{t}^{+\infty}\bar{F}(x)dx.
\end{eqnarray*}

Now, we can define a generalized measure of cumulative residual extropy as
\begin{eqnarray}
\mathcal{\xi J}(X;t)=-\frac{1}{2} \int_{t}^{+\infty}\left[\frac{\bar{F}(x)}{\bar{F}(t)}\right]%
^{2}dx.
\end{eqnarray}
Note that $\mathcal{\xi J}(X;t)\geq -\mu (t)/(2\bar{F}(t))$. Moreover , we have
\begin{equation} \label{HX_SRS}
{\mathcal{\xi J}}(\boldsymbol{X}^{(m)}_{SRS};t)=-\frac{1}{2}[-2 {\mathcal{\xi J}}%
(X,t)]^{m}.
\end{equation}
Under the MinRSSU design, it is clear that
\begin{eqnarray} \label{HX_MRSSU}
{\mathcal{\xi J}}(\boldsymbol{X}^{(m)}_{MinRSSU};t)&=&-\frac{1}{2}\prod
\limits_{i=1}^{m}[-2{\mathcal{\xi J}}(X_{(i:i)};t)]
= -\frac{1}{2}\prod \limits_{i=1}^{m}\int_{t}^{+\infty}\left[\frac{\bar{F}(x)}{\bar{F}(t)}%
\right]^{2i}dx \notag \\
&=& -\frac{1}{2}\prod \limits_{i=1}^{m}\mathbb{E}\left[\frac{U^{2i}\bar{F}(t)}{%
f(F^{-1}(1-U\bar{F}(t)))}\right].
\end{eqnarray}

\begin{theorem}
Let $X$ be a random lifetime variable with cdf $F(\cdot)$,
Then, for $m>1$
\begin{eqnarray}
\mathcal{\xi J}(\boldsymbol{X}^{(m)}_{MinRSSU};t)\leq\mathcal{\xi J}(\boldsymbol{X}%
^{(m)}_{SRS};t).
\end{eqnarray}
\end{theorem}
\begin{proof}
The proof is similar to 
Theorem \ref{th-2-1}.
\end{proof}

\begin{remark}
If $f(F^{-1}(1-u\bar{F}(t)))\geq1$, $0<u<1$, then $\mathcal{\xi J}(\boldsymbol{X}^{(m)}_{MinRSSU};t)$ is increasing in $m\geq 1$.
\end{remark}

\section{Results on empirical measure of CREX}\label{sec-empirical}

This section focuses on the estimation of the ${\mathcal{\xi J}}(X)$
based on the SRS and MinRSSU schemes. Let \bigskip $X_{(1)}\leq X_{(2)}\leq
...\leq X_{(n)}$ be the order statistics of the random sample $%
X_{1},X_{2},...,X_{n}$ from cdf $F.$ Then the empirical measure of $F$ is
defined as
\begin{equation*}
\hat{F}_{n}(x)=\left\{
\begin{array}{ll}
0, & \ x<X_{(1)}, \\
\frac{k}{n}, & X_{(k)}\leq x\leq X_{(k+1)},\ \ \ \ k=1,2,...,n-1 \\
1, & \ x>X_{(n)}.%
\end{array}%
\right.
\end{equation*}

Thus, the empirical measure of ${\mathcal{\xi J}}(X)$ is obtained by
replacing the distribution function F by the empirical distribution function
$\hat{F}_{n}$ as
\begin{eqnarray}
V_{n} &=&-\frac{1}{2}\int_{{}}^{{}}\hat{\bar{F}}_{n}^{2}(x)dx=-\frac{1}{2}%
\sum_{k=1}^{n-1}\int_{X_{(k)}}^{X_{(k+1)}}\left( 1-\frac{k}{n}\right) ^{2}dx
\notag \\
&=&-\frac{1}{2}\sum_{k=1}^{n-1}U_{k+1}\left( 1-\frac{k}{n}\right) ^{2},
\end{eqnarray}%
where $U_{k+1}=X_{(k+1)}-X_{(k)}, k=1,...,n-1$. \citet{ja-et-al-19} showed
that $V_{n}$ almost surely converges to the CREX of $X$, i.e.
\begin{equation*}
V_{n}\overset{a.s.}{\rightarrow }{\mathcal{\xi J}}(X),\;\; \text{as }%
\; n\rightarrow +\infty .
\end{equation*}
The problem of estimation of ${\mathcal{\xi J}}(X)$\ based on MinSSU scheme
can be deduced in the same line of the estimation based on SRS design $V_{n}$%
. In this part, we assume that instead of one-cycle MinRSSU, the process is
repeated $l$ cycles to have a sample of size $n=ml$. In this case, the resulting MinRSSU is denoted by $\left\{ X_{\left( 1:i\right) j},\text{ }%
i=1,...,m;\text{ }j=1,...,l\right\} $, where $X_{\left( 1:i\right) j}$ is
the lowest order statistic from the $i$th sample in the $j$th cycle. Let $%
Y_{\left( 1\right) },...,Y_{\left( n\right) }$ be the ordered values of the
MinRSSU design $\left\{ X_{\left( 1:i\right) j},\text{ }i=1,...,m;\text{ }%
j=1,...,l\right\} .$ Then, the natural estimation of ${\mathcal{\xi J}}(X)$
based on the MinRSSU, can be obtained as
\begin{equation*}
R_{n}=-\frac{1}{2}\sum_{k=1}^{n-1}Z_{k+1}\left( 1-\frac{k}{n}\right) ^{2},
\end{equation*}
where $Z_{k+1}=Y_{(k+1)}-Y_{(k)},k=1,2,...,n-1$. Our preliminary
computations and simulations showed that $R_{n}$ has some deficiencies to
be unbiased and have low mean square error (MSE) to estimate the ${\mathcal{%
\xi J}}(X).$ We observed that the term $\left( 1-\frac{k}{n}\right) $ should
be slightly modified so that the estimator has optimal properties to estimate $%
{\mathcal{\xi J}}(X).$ We propose to modify this term with $\left( 1-%
\frac{k}{n+m+w}\right),$ where $w$ is a number that resulting estimator has
optimally low bias and MSE. The resulting estimator of the ${\mathcal{\xi J}}%
(X)$ has the following form
\begin{equation} \label{eq-r-mn-2}
R_{m,n}=-\frac{1}{2}\sum_{k=1}^{n-1}Z_{k+1}\left( 1-\frac{k}{n+m+w}\right)
^{2}.
\end{equation}
\subsection{A new estimation}
In the previous section, the estimator $V_{n}$ of ${\mathcal{\xi J}}(X)$ is a linear function of sample spacing $U_{k+1}=X_{(k+1)}-X_{(k)},k=1,...,n-1$. The asymptotic distribution of this linear function of
sample spacing can be found, for example in \citet{dic-long-09}
and \citet{tahmasebi-2019} only for exponential and standard uniform
distributions. So we provide another estimator for ${\mathcal{\xi J}}(X)$ which is a linear function of order statistics.
\begin{proposition}
Let $X$ be an absolutely continuous non-negative random variable with survival function $\bar{F}$, then
\begin{equation}\label{eq-another-formula-cre}
{\mathcal{\xi J}}(X)=-\int^{+\infty}_{0}x\bar{F}(x)d{F}(x).
\end{equation}
\begin{proof}
By \eqref{JX} and Fubini's theorem, we obtain
\begin{eqnarray*}
-\int^{+\infty}_{0}x\bar{F}(x)d{F}(x)&=&-\int^{+\infty}_{0}\left(\int^{x}_{0}dt\right)\bar{F}(x)f(x)dx
\notag \\
&=&-\int^{+\infty}_{0}\left(\int^{+\infty}_{t}\bar{F}(x)f(x)dx\right)dt=-\frac{1}{2}\int_{0}^{+\infty}\bar{F}^{2}(t)dt.
\end{eqnarray*}
Hence, the proof is completed
\end{proof}
\end{proposition}

The new estimator can be obtained replacing $\bar{F}(x)$ with $\hat{\bar{F}}_{n}(x)$ in \eqref{eq-another-formula-cre}, so $\widehat{\mathcal{\xi J}}(X)$ has the following form%

\begin{equation}\label{eq-another-estimate}
\widehat{\mathcal{\xi J}}(X)=-\int_{0}^{+\infty }x\hat{\bar{F%
}}_{n}\left( x\right) d\hat{F}_{n}\left( x\right) =-\frac{1}{n}%
\sum_{i=1}^{n}\left( 1-\frac{i}{n}\right) X_{\left( i\right) }.
\end{equation}

Let $J\left( x\right)=\left( 1-x\right).$   Then \eqref{eq-another-estimate} has the form
\begin{equation*}
\widehat{{\mathcal{\xi J}}}(X)=-\frac{1}{n}\sum_{i=1}^{n}J\left( \frac{i}{n%
}\right) X_{\left( i\right) },
\end{equation*}
which is a linear function of order statistics.
The natural estimation of ${\mathcal{\xi J}}(X)$
based on the MinRSSU, can be obtained as
\begin{equation*}
\widehat{{\mathcal{\xi J}}}(Y):=\widehat{{\mathcal{\xi J}}}(X_{MinRSSU})=-\frac{1}{n}\sum_{i=1}^{n}J\left( \frac{i}{n%
}\right) Y_{\left( i\right) }.
\end{equation*}
\citet{stigler-74} showed that asymptotic distribution of such a linear combination is normal distribution. The results of \citet{stigler-74} also hold if the
independent observations are not identically distributed. These properties
help us to obtain the asymptotic distribution of $\widehat{{\mathcal{\xi J}}}(X)$
for both SRS (observations are independent and identical) and MinRSSU (observations are only independent) designs.
\begin{theorem}\label{th-asym-estimator}
Assume that $E\left( X^{2}\right) <+\infty $. Then
\begin{eqnarray*}
&&\sqrt{n}\left( \widehat{{\mathcal{\xi J}}}(X)-{\mathcal{\xi J}}%
(X)\right) \overset{d}{\rightarrow }N\left( 0,\sigma ^{2}\left( J,F\right)
\right) , \\
&&\sqrt{n}\left( \widehat{{\mathcal{\xi J}}}(Y)-{\mathcal{\xi J}%
}(X)\right) \overset{d}{\rightarrow }N\left( 0,\sigma _{MinRSSU}^{2}(
 J,\tilde{F},K) \right) ,
\end{eqnarray*}
where
\begin{equation*}
\sigma ^{2}\left( J,F\right) =\int_{0}^{+\infty}\int_{0}^{+\infty}J(F(x)) J(F(y)) [F(min(x,y))-F(x)F(y)] dx dy,
\end{equation*}
\begin{equation*}
\sigma _{MinRSSU}^{2}( J,\tilde{F},K) =\int_{0}^{+\infty}\int_{0}^{+\infty}J(\tilde{F}(x)) J(\tilde{F}(y))K(x,y)dx dy,
\end{equation*}
and $\tilde{F}(.)$ and $K(x,y)$ are given as
\begin{eqnarray*}
\tilde{F}(x)&=&\frac{1}{m}\sum_{i=1}^{m}F_{(1)i}(x),  \\
K(x,y)&=&\frac{1}{m}\sum_{i=1}^{m} [F_{(1)i}(min(x,y))-F_{(1)i}(x)F_{(1)i}(y)].
\end{eqnarray*}
\end{theorem}
\begin{proof}
For $\tilde{F}(x)$ and $K(x,y)$, we have
\begin{eqnarray*}
\tilde{F}(x)&=&\lim_{n\to +\infty}\frac{l}{n}\sum_{i=1}^{m}F_{(1)i}(x)=\frac{1}{m}\sum_{i=1}^{m}F_{(1)i}(x),  \\
K(x,y)&=&\lim_{n\to +\infty}\frac{l}{n}\sum_{i=1}^{m} [F_{(1)i}(min(x,y))-F_{(1)i}(x)F_{(1)i}(y)]\\
         &=&\frac{1}{m}\sum_{i=1}^{m} [F_{(1)i}(min(x,y))-F_{(1)i}(x)F_{(1)i}(y)],
\end{eqnarray*}
where as before $l$ is the size of the cycle of MinRSSU design with $n=ml$.
The rest of the proof is done by using the results of \citet{stigler-74} for both SRS and MinRSSU designs.
\end{proof}
As we formerly stated for providing the estimator $R_{m,n}$, some adjusted forms of estimator $\widehat{{\mathcal{\xi J}}}(Y)$ can be used since the estimator $\widehat{{\mathcal{\xi J}}}(Y)$ has some deficiencies to estimate CREX which need to be fixed. We again observed that the
the term $\left( 1-\frac{i}{n} \right)$ needs to be slightly adjusted so that the resulting estimator
has optimal properties to estimate CREX. This new term has the following form
\begin{equation}
1-\frac{i}{n+\psi(m,w)},
\end{equation}
where the function $\psi(.,.)$  is a challenging  factor that can be specifically determined for each given distribution and consequently the resulting estimator has optimal low bias and MSE.
In this case, the resulting estimator has the following form
\begin{equation}\label{eq-second-form-second-esti}
\widehat{{\mathcal{\xi J}}}_{m,n}(Y)=-\frac{1}{n}\sum_{i=1}^{n}J\left( \frac{i}{n+\psi(m,w)%
}\right) Y_{\left( i\right) }.
\end{equation}
In the next section, we determine the form of the function $\psi(.,.)$ for exponential, uniform and beta distributions and we show that the optimal choice of this function can reduce the bias and MSE in the estimate of CREX.
\subsection{Simulation study}
In advance, we explain the role of the parameter $w$ and function $\psi(m,w)$ for which $R_{m,n}$ and $\widehat{{\mathcal{\xi J}}}_{m,n}(X_{MinRSSU})$ have optimally low bias and MSE. For this purpose, we examine some distributions
to obtain the function $\psi(m,w)$ and optimal value of $w$ in $R_{m,n}$ and $\widehat{{\mathcal{\xi J}}}_{m,n}(Y)$.
In Tables \ref{tab-exp}-\ref{tab-beta}, for four
estimators $R_{n}$,  $R_{m,n}$, $\widehat{{\mathcal{\xi J}}}(Y)$ and $\widehat{{\mathcal{\xi J}}}_{m,n}(Y)$, we compute the bias and MSE to estimate the parameter ${\mathcal{\xi J}}(X)$ for some different values of $w$.
Here, the exponential ($Exp\left( \lambda\right) $)$,$ uniform ($Unif\left( 0,b\right) $) and beta ($Beta\left( \alpha,1 \right), \alpha>1 $) distributions are considered.
For each configuration, the simulation study was carried out with 5000
repetitions. The number of cycle and size of the sample in
each cycle are taken as $l=2,3$ and $m=2,...,5$, respectively. We compute
the bias and root of MSE (RMSE) of each estimator of parameter ${\mathcal{\xi J}%
}(X)$. In Tables \ref{tab-exp}-\ref{tab-beta}, it can be seen that results of biases and RMSEs of estimator
$\widehat{{\mathcal{\xi J}}}(Y)$ are not comparable to those of $\widehat{{\mathcal{\xi J}}}_{m,n}(Y)$
for different values of $w$. We intuitively obtain the form of the function $\psi(m,w)$ for different distributions. We found that the function $\psi$ has the form $5m-4k_{m}+w$ with $k_{2}=3$,..., $k_{5}=0$, $3m-(2k_{m}+1)+w$ with $k_{2}=-1$,..., $k_{5}=2$ and $m-w$ for exponential, uniform and beta distributions, respectively. It is observed that changing the value of $l$ has no effect on the whole results verified from biases and RMSEs. In all tables, we see that choosing the proper function $\psi(m,w)$ and parameter $w$ as the challenging factors can improve the efficiency of estimators $R_{m,n}$ and $\widehat{{\mathcal{\xi J}}}_{m,n}(Y)$ against $R_{n}$ and $\widehat{{\mathcal{\xi J}}}(Y)$, respectively, in estimating the parameter $\mathcal{\xi J}$.
\begin{table}[ht]
\scriptsize
\centering
\caption{The biases and MSEs of the different estimators: Exponential distribution}\label{tab-exp}
\resizebox{\textwidth}{!}{
\begin{tabular}{cccccccc|ccccccc}\hline 
&&\multicolumn{5}{c}{Results based on $R_{m,n}$}     &&   &&\multicolumn{5}{c}{Results based on $\widehat{\mathcal{\xi J}}_{m,n}(Y)$}\\   \cline{3-7} \cline{11-15}
& &  $l=2$ & & & $l=3$ &         &&    &  &  $l=2$& & & $l=3$ &  \\ \cline{3-4} \cline{6-7}\cline{11-12} \cline{14-15}
$m$ & $w$  &  Bias & RMSE  &&  Bias & RMSE && $m$ & $w$  &  Bias & RMSE &&  Bias & RMSE \\
2 & -2 & 0.321 &0.407   &&    0.370 &0.477  &&   2 & -11 & 0.632& 0.650   &&   0.536& 0.562  \\
   & -1 & 0.131 &0.354   &&   0.133  &0.460      &&          & -10 &  0.402& 0.483   &&    0.354& 0.427\\
   &  0 &-0.051 &0.423   &&   -0.069 &0.580      &&          & -9 &  0.263& 0.435   &&   0.223& 0.370\\
   &  1 &-0.216 &0.556   &&   -0.234  &0.733     &&          & -8 &  0.171& 0.434   &&   0.125& 0.361\\
   &$R_{n}$& 0.321 &0.407   &&    0.370  &0.477     &&          & $\widehat{{\mathcal{\xi J}}}(Y)$ &  0.402& 0.483   &&   0.354& 0.427\\
3 & -1 & 0.124 &0.367   &&   0.164   &0.304     &&        3 & -7 &  0.033& 0.374   &&   0.006& 0.307\\
   &  0 &-0.017 &0.429   &&   0.040   &0.315     &&          & -6 &  0.012& 0.382   &&   -0.020&  0.317\\
   &  1 &0.124 &0.367    &&   0.164   &0.304     &&          & -5 &  -0.006&  0.391   &&   -0.044&  0.328\\
   &  2 &-0.253 &0.632   &&   -0.191  &0.472     &&          & -4 &  -0.022&  0.399   &&   -0.065&  0.340\\
   & $R_{n}$&0.437 &0.478    &&   0.411   &0.444     &&          &  $\widehat{{\mathcal{\xi J}}}(Y)$ &    0.452& 0.489   &&     0.430& 0.458\\
4 & 0  &0.024 &0.383    &&   0.118   &0.291     &&        4 & -3 & 0.029& 0.329   &&   -0.007&  0.269\\
   & 1 &-0.077 &0.452    &&   0.029   &0.313     &&          & -2 &  0.020& 0.332   &&   -0.019 & 0.274\\
   & 2 &-0.170 &0.532    &&   -0.057  &0.363     &&          & -1 &   0.012& 0.334   &&    -0.030  &0.278\\
   & 3 &-0.256 &0.616    &&   -0.139  & 0.427    &&          & 0 &  0.004& 0.337   &&   -0.040 & 0.283\\
   &$R_{n}$ &0.486 &0.509    &&   0.474   &0.490     &&          &  $\widehat{{\mathcal{\xi J}}}(Y)$ &   0.509& 0.527   &&   0.489 &0.503\\
5 & 1 &-0.004 &0.390    &&   0.114   &0.281     &&        5 &1 &  0.072 &0.277   &&   0.050 & 0.232\\
   & 2 &-0.082 &0.448    &&   0.046   &0.296     &&          & 2 &  0.067& 0.277   &&   0.043& 0.233\\
   & 3 &-0.156 &0.512    &&   -0.021  &0.329     &&          & 3 & 0.063& 0.278  &&   0.036 &0.233\\
   & 4 &-0.224 &0.578    &&   -0.085  &0.373     &&          & 4 &  0.058& 0.278   &&   0.030& 0.234\\
   &$R_{n}$&0.533  &0.545    &&   0.524   &0.532     &&         &  $\widehat{{\mathcal{\xi J}}}(Y)$ &  0.554 &0.563   &&   0.540 &0.547\\   \hline
\end{tabular}
}
\end{table}
\begin{table}[ht]
\scriptsize
\centering
\caption{The biases and MSEs of the different estimators: Uniform distribution}\label{tab-unif}
\resizebox{\textwidth}{!}{
\begin{tabular}{cccccccc|ccccccc}\hline 
&&\multicolumn{5}{c}{Results based on $R_{m,n}$}     &&   &&\multicolumn{5}{c}{Results based on $\widehat{\mathcal{\xi J}}_{m,n}(Y)$}\\   \cline{3-7} \cline{11-15}
& &  $l=2$ & & & $l=3$ &         &&    &  &  $l=2$& & & $l=3$ &  \\ \cline{3-4} \cline{6-7}\cline{11-12} \cline{14-15}
$m$ & $w$  &  Bias & RMSE  &&  Bias & RMSE   &&   $m$ & $w$  &  Bias & RMSE &&  Bias & RMSE \\
2 & -2 & 0.360 &0.392 &&  0.298 &0.333  &&        2 & -4 &  0.098 &0.275   &&   0.058& 0.227  \\
& -1 & 0.249 &0.305 &&  0.212 &0.264  &&          & -3 &  0.045 &0.281   &&   -0.001  &0.236  \\
& 0 & 0.158 &0.253 &&  0.134 &0.214  &&          & -2 &  0.005 &0.293   &&   -0.049  &0.254  \\
& 1 & 0.085 &0.233 &&  0.066 &0.186  &&          & -1 &  -0.026  &0.306   &&   -0.088 & 0.275  \\
& $R_{n}$& 0.360 &0.392 && 0.298 &0.333  &&          &  $\widehat{{\mathcal{\xi J}}}(Y)$ &  0.284 &0.340   &&   0.238 &0.293  \\
3 & -1 &0.191 &0.243 && 0.201 &0.237  &&        3 & -2 &  -0.006  &0.247   &&   -0.023 & 0.200  \\
& 0 &0.125 &0.204 && 0.147 &0.196  &&              & -1 &  -0.026  &0.255   &&   -0.047  &0.209  \\
& 1 &0.067 &0.183 && 0.096 &0.165  &&              & 0 & -0.043  &0.263   &&   -0.068  &0.220  \\
& 2 &0.016 &0.181 && 0.055 &0.148  &&             & 1  &   -0.058  &0.270   &&   -0.087 & 0.231  \\
&$R_{n}$ &0.349 &0.372 && 0.318 &0.339  &&            & $\widehat{{\mathcal{\xi J}}}(Y)$ &  0.318& 0.350   &&  0.290 &0.317 \\
4 & 0 &0.14 &0.195&& 0.175 &0.204  &&        4 & 0 &  -0.012 & 0.205   &&   -0.008&  0.164  \\
& 1 &0.092 &0.171 && 0.135 &0.173  &&          & 1  &  -0.023  &0.209   &&   -0.021 & 0.168  \\
& 2 &0.049 &0.159&& 0.098 &0.148  &&          &  2 &  -0.034 & 0.213   &&    -0.034 & 0.172  \\
& 3 &0.018 &0.159 && 0.062 &0.131  &&          & 3 &  -0.043&  0.217   &&   -0.046 & 0.177  \\
&$R_{n}$&0.372 &0.387 && 0.349 &0.360  &&          &  $\widehat{{\mathcal{\xi J}}}(Y)$ &  0.352 &0.369   &&   0.337 &0.350  \\
5 & 1&0.116 &0.17&& 0.175 &0.198  &&        5 & 2 &  0.020& 0.181   &&   0.028& 0.146  \\
& 2 &0.079 &0.153 && 0.144 &0.172  &&          & 3 &  0.013 &0.183   &&   0.019 &0.146  \\
& 3 &0.044 &0.144 && 0.114 &0.150  &&          & 4 &  0.006& 0.184   &&   0.010 &0.147 \\
& 4 &0.012 &0.144 && 0.085 &0.132  &&          & 5 &  -0.001&  0.186   &&   0.002& 0.148  \\
&$R_{n}$&0.392 &0.401 && 0.381 &0.388  &&          &  $\widehat{{\mathcal{\xi J}}}(Y)$ &  0.387& 0.398   &&   0.375 &0.383  \\   \hline
\end{tabular}
}
\end{table}
\begin{table}[ht]
\scriptsize
\centering
\caption{The biases and MSEs of the different estimators: Beta distribution}\label{tab-beta}
\resizebox{\textwidth}{!}{
\begin{tabular}{cccccccc|ccccccc}\hline 
&&\multicolumn{5}{c}{Results based on $R_{m,n}$}     &&   &&\multicolumn{5}{c}{Results based on $\widehat{\mathcal{\xi J}}_{m,n}(Y)$}\\   \cline{3-7} \cline{11-15}
& &  $l=2$ & & & $l=3$ &         &&    &  &  $l=2$& & & $l=3$ &  \\ \cline{3-4} \cline{6-7}\cline{11-12} \cline{14-15}
$m$ & $w$  &  Bias & RMSE  &&  Bias & RMSE   &&   $m$ & $w$  &  Bias & RMSE &&  Bias & RMSE \\
2 & -2 & 0.296 & 0.298 &&  0.276 &0.279 &&        2 & -3 &  0.202 &0.204   &&   0.146 &0.149  \\
& -1 & 0.280 & 0.283 &&  0.264 &0.267 &&                & -2 &  0.108 &0.116   &&   0.081 &0.089  \\
& 0 & 0.267 & 0.271 &&  0.253 &0.257 &&                   & -1 &  0.053 &0.070   &&   0.035 &0.055  \\
& 1 & 0.257 & 0.261 &&  0.244 &0.248 &&                   & 0 &  0.015 &0.053   &&   0.000 &0.045  \\
& $R_{n}$& 0.296 & 0.298 &&  0.276 &0.279 &&                 &  $\widehat{{\mathcal{\xi J}}}(Y)$ &   0.108 &0.116   &&   0.081 &0.089  \\
3 & -1 &0.251 &0.255& & 0.238 &0.242 &&           3 & -3 &  0.096 &0.102   &&   0.079 &0.085  \\
& 0 &0.241 &0.245 & & 0.230 &0.234 &&                    & -2 &  0.052 &0.064   &&   0.045 &0.056  \\
& 1 &0.233 &0.238 && 0.222 &0.227 &&                     & -1 &  0.019 &0.045   &&   0.017 &0.039  \\
& 2 &0.225 &0.231 && 0.216 &0.221 &&                     & 0 &  -0.007 & 0.043   &&   -0.006  &0.037  \\
&$R_{n}$&0.275 &0.278 & & 0.256 &0.259 &&                   & $\widehat{{\mathcal{\xi J}}}(Y)$ &  0.096& 0.102   &&   0.079 &0.085  \\
4 & 0 &0.227 &0.231& & 0.218 &0.222 &&              4 & -3 &  0.059 &0.068   &&   0.056& 0.063  \\
& 1 &0.220 &0.225 & & 0.212 &0.216 &&                     & -2 &  0.031& 0.047   &&   0.033& 0.044  \\
& 2 &0.214 &0.219& & 0.206 &0.211 &&                     & -1 & 0.007 &0.038   &&   0.013 &0.033  \\
& 3 &0.208 &0.213 & & 0.201 &0.206 &&                    & 0 &  -0.012  &0.041   &&   -0.004 & 0.031  \\
&$R_{n}$&0.264 &0.267 & & 0.246 &0.249 &&                    &  $\widehat{{\mathcal{\xi J}}}(Y)$ &  0.095 &0.100   &&   0.082 &0.087  \\
5 & 1 &0.217 &0.221& & 0.210 &0.214 &&        5 & -3 & 0.043& 0.053   &&   0.047 &0.054  \\
& 2 &0.211 &0.216 & & 0.205 &0.209 &&                & -2 &  0.022& 0.039   &&   0.030& 0.040  \\
& 3 &0.206 &0.210 & & 0.200 &0.204 &&                & -1 &  0.004& 0.034   &&   0.015 &0.031  \\
& 4 &0.201 &0.205 & & 0.196 &0.201 &&                & 0 &  -0.012 & 0.036   &&   0.002 &0.028  \\
&$R_{n}$&0.257 &0.259 & &0.239 &0.242 &&                & $\widehat{{\mathcal{\xi J}}}(Y)$ &  0.097 &0.101   &&   0.088 &0.091  \\  \hline
\end{tabular}
}
\end{table}
\section{Discrimination information} \label{sec-discimin}
This section considers a new discrimination measure of disparity between the
distribution of MinRSSU and parental data SRS. \citet{raq-qiu-19} defined
the discrimination information between the density function of the $i$th
order statistic $f_{(i)m}$ and the underlying density function $f$ as
\begin{equation}\label{eq-raq-qiu-disc}
D_{m}\left( f_{(i)m}:f\right) =\frac{1}{2}\int_{-\infty }^{+\infty
}f_{(i)m}\left( x\right) \left( f_{(i)m}\left( x\right) -f\left( x\right)
\right) dx.
\end{equation}
Analogously to \eqref{eq-raq-qiu-disc}, we define the discrimination
information between the survival function of the smallest ordered statistic $%
\bar{F}_{(1)i}$ and the underlying survival function $\bar{F}$ as
\begin{equation}\label{eq-new-disc}
D\left( \bar{F}_{(1)i}:\bar{F}\right) =-\frac{1}{2}\int_{-\infty }^{+\infty }%
\bar{F}_{(1)i}\left( x\right) \left( \bar{F}_{(1)i}\left( x\right) -\bar{F}%
\left( x\right) \right) dx.
\end{equation}
The discrimination information in \eqref{eq-new-disc} may be rewritten in
a simpler way. It can be shown that
\begin{equation}\label{eq-new-disc-1}
D\left( \bar{F}_{(1)i}:\bar{F}\right) =-\frac{1}{2}\left[ E( X_{(
1) 2i}) -E( X_{( 1) i+1}) \right] ,
\end{equation}
where $X_{( 1) j}$ is the smallest order statistic in a random
sample of size $j$.
\begin{example} \label{ex-unif-disc}
Let $U\sim Uniform(0,1)$. We know that the order
statistics from the standard uniform distribution follow the beta
distribution. Then discrimination information based on \eqref{eq-new-disc} is
\begin{equation*}
D\left( \bar{F}_{(1)i}:\bar{F}\right) =-\frac{1}{2}\left[ \frac{1}{2i+1}-%
\frac{1}{i+2}\right] =\frac{i-1}{2\left( 2i+1\right) \left( i+2\right) }.
\end{equation*}
\end{example}
In the following theorem, we obtain the discrimination information $D$
between MinRSSU and SRS designs.

\begin{theorem} \label{th-disc}
For $\boldsymbol{X}_{MinRSSU}^{(m)}$ and $\boldsymbol{X}%
_{SRS}^{(m)}$, we have
\begin{equation*}
D\left( \boldsymbol{X}_{MinRSSU}^{(m)}:\boldsymbol{X}_{SRS}^{(m)}\right) =-%
\frac{1}{2}\left( \prod\limits_{i=1}^{m}E\left( X_{\left( 1\right)
2i}\right) -\prod\limits_{i=1}^{m}E\left( X_{\left( 1\right) i+1}\right)
\right) .
\end{equation*}
\end{theorem}
\begin{proof} From (2.1) and (4.2) we have
\begin{eqnarray*}
D\left( \boldsymbol{X}_{MinRSSU}^{(m)}:\boldsymbol{X}_{SRS}^{(m)}\right)
&=&-\frac{1}{2}\left( \prod\limits_{i=1}^{m}\int_{-\infty }^{+\infty }\bar{F}%
^{2i}\left( x_{1}\right) -\prod\limits_{i=1}^{m}\int_{-\infty }^{+\infty }%
\bar{F}^{i+1}\left( x_{1}\right) \right) \\
&=&-\frac{1}{2}\left(\prod\limits_{i=1}^{m}E\left( X_{\left( 1\right) 2i}\right)
-\prod\limits_{i=1}^{m}E\left( X_{\left( 1\right) i+1}\right) \right) .
\end{eqnarray*}
The proof is completed.
\end{proof}
\begin{example} \label{ex-unif-rss-disc}
Let $U\sim Uniform(0,1)$. Using the results of
Example \eqref{ex-unif-disc}, the discrimination information $D$ for the
MinRSSU and SRS designs of the same size $m$ is
\begin{equation*}
D\left( \boldsymbol{U}_{MinRSSU}^{(m)}:\boldsymbol{U}_{SRS}^{(m)}\right) =-%
\frac{1}{2}\left( \prod\limits_{i=1}^{m}\frac{1}{2i+1}-\prod\limits_{i=1}^{m}%
\frac{1}{i+2}\right) .
\end{equation*}
\end{example}
\section{Conclusion}\label{sec-conclude}
This paper has introduced the uncertainty measure of the cumulative residual extropy based on the MinRSSU and SRS data. Several results of the CREX measure including stochastic orders were obtained for MinRSSU and SRS data. Also, we provided two estimators of CREX measure for both SRS and MinRSSU data. Furthermore, it was shown that MinRSSU scheme can efficiently reduce the uncertainty measure of CREX. Also, by providing a discrimination measure, we derived the distance size between MinRSSU and SRS data.

\section*{Acknowledgements}
C. Cal\`{i} and M. Longobardi are partially supported by the GNAMPA research group of INDAM (Istituto Nazionale di Alta Matematica) and MIUR-PRIN 2017, Project "Stochastic Models for Complex Systems" (No. 2017JFFHSH).
\section*{References}

\bibliography{mybibfile}

\end{document}